\documentclass[11pt, reqno]{amsart}

\usepackage{amsmath, amsfonts, amssymb, amsbsy, bigstrut, graphicx, enumerate,  upref, longtable, comment, booktabs, array, caption, subcaption}

\usepackage{pstricks}
\usepackage{times}

\usepackage[breaklinks]{hyperref}

\usepackage[numbers, sort&compress]{natbib}

\newcommand{\ep}{\epsilon}

\newcommand{\me}{\mathcal{E}}

\newtheorem{thm}{Theorem}[section]
\newtheorem{lmm}[thm]{Lemma}
\newtheorem{cor}[thm]{Corollary}
\newtheorem{prop}[thm]{Proposition}

\theoremstyle{definition}

\newcommand{\ee}{\mathbb{E}}

\newcommand{\ma}{\mathcal{A}}

\newcommand{\mx}{\mathcal{X}}

\newcommand{\pp}{\mathbb{P}}

\newcommand{\rr}{\mathbb{R}}

\newcommand{\zz}{\mathbb{Z}}

\numberwithin{equation}{section}

\usepackage[utf8]{inputenc}
\usepackage{amsmath}
\usepackage{amsfonts}
\usepackage{bbm}
\usepackage{mathtools}
\usepackage{tikz}
\usepackage{amsthm}
\usepackage[english]{babel}
\usepackage{fancyhdr}
\usepackage{amssymb}
\usepackage{enumitem}
\usepackage{qtree}
\usepackage{mathrsfs}

\renewcommand{\hat}{\widehat}

\begin{document}

\title{Speeding up Markov chains with deterministic jumps}
\author{Sourav Chatterjee}
\address{Departments of mathematics and statistics, Stanford University, USA}
\email{souravc@stanford.edu}
\email{diaconis@math.stanford.edu}
\thanks{Sourav Chatterjee's research was partially supported by NSF grant DMS-1855484}
\author{Persi Diaconis}
\thanks{Persi Diaconis's research was partially supported by NSF grant DMS-1954042}
\thanks{Data availability statement: Data sharing not applicable to this article as no datasets were generated or analyzed during the current study}

\dedicatory{In memory of Harry Kesten}

\begin{abstract}
We show that the convergence of finite state space Markov chains to stationarity can often be considerably  speeded up by alternating every step of the chain with a deterministic move. Under fairly general conditions, we show that not only do such schemes exist, they are numerous.
\end{abstract} 

\maketitle


\section{Introduction}\label{intro}
This paper started from the following example. Consider the simple random walk on $\zz_n$ (the integers mod $n$): 
\[
X_{k+1}=X_k+\ep_{k+1} \pmod n,
\]
with $X_0=0$, and $\ep_1,\ep_2,\ldots$ i.i.d.~with equal probabilities of being $0$, $1$ or $-1$. This walk takes order $n^2$ steps to reach its uniform stationary distribution in total variation distance. This slow, diffusive behavior is typical of low dimensional Markov chains (``Diameter$^2$ steps are necessary and sufficient'' --- see \citet{dsc} for careful statements and many examples). Now change the random walk by deterministic doubling:
\begin{align*}
X_{k+1} = 2X_k + \ep_{k+1} \pmod n.
\end{align*}
This walk has the ``same amount of randomness''. Results discussed below show that it takes order $\log n$ steps to mix (at least for almost all $n$). 

We would like to understand this speedup more abstractly --- hopefully to be able to speed up real world Markov chains. Our main result shows a similar speedup occurs for fairly general Markov chains with deterministic doubling replaced by almost any bijection of the state space into itself. We proceed to a careful statement in the next section. In Section \ref{review} a literature review offers pointers to related efforts to beat diffusive behavior. A sketch of the proof is in Section \ref{sketchsec}. Proofs are in Sections  \ref{proofsec} and \ref{proof2}. A different kind of deterministic speedup,
\[
X_{k+1} = X_k + X_{k-1} + \ep_{k+1}\pmod n,
\]
is analyzed in Section \ref{fibosec}. The last section has applications and some interesting open questions. 

\section{Speedup by deterministic functions}
Let $S$ be a finite set and $P = (p_{ij})_{i,j\in S}$ be a Markov transition matrix on $S$. We assume the following conditions on $P$:
\begin{enumerate}
\item The Markov chain generated by $P$ is irreducible and aperiodic.
\item If $p_{ij}>0$ for some $i,j\in S$, then $p_{ji}>0$.
\item For each $i\in S$, $p_{ii}>0$.
\item The uniform distribution on $S$  is the stationary measure of $P$. We will call it $\mu$.
\end{enumerate}
In addition to the above assumptions, we define  
\[
n := |S|
\]
and 
\begin{align}\label{deltadef}
\delta := \min\{p_{ij}: i,j\in S, \, p_{ij}>0\}.
\end{align}
Let $f:S\to S$ be a bijection. Consider a new Markov chain defined as follows. Apply $f$, and then take one step according to $P$. This is one step of the new chain. The transition matrix for the new chain is
\[
Q = \Pi P,
\]
where $\Pi = (\pi_{ij})_{i,j\in S}$ is the permutation matrix defined by $f$. (That is, $\pi_{ij}=1$ if $j=f(i)$ and $0$ otherwise.)  Clearly, $\mu$ is a stationary measure for the new chain too. The main questions that we will try to answer are the following: (a) Under what conditions on $f$ does this new chain mix quickly --- let's say, in time $\log n$? (b) Is it always possible to find $f$ satisfying these conditions? (c) If so, are such functions rare or commonplace?

Our first theorem answers question (a), by giving a sufficient condition for fast mixing of $Q$.
\begin{thm}\label{thm1}
For any $A\subseteq S$, let $\me(A)$ be the set of $j\in S$ such that $p_{ij}>0$ for some $i\in A$. Suppose that there is some $\ep\in (0,1)$ such that for any $A$ with $|A|\le n/2$,
\begin{equation}\label{expcond}
|\me\circ f\circ \me(A)|\ge (1+\ep)|A|.
\end{equation}
Let $X_0,X_1,\ldots$ be a Markov chain with transition matrix $Q$. For each $i\in S$ and $k\ge 1$ let $\mu_i^k$ be the law of $X_k$ given $X_0=i$. Then 
\[
\|\mu_i^k- \mu\|_{TV} \le \frac{\sqrt{n}}{2} \biggl(1-\frac{\ep^2\delta^8}{2}\biggr)^{(k-2)/4},
\]
where $\delta$ is the quantity defined in \eqref{deltadef} and $TV$ stands for total variation distance.
\end{thm}
This result shows that if we have a sequence of problems where $n\to\infty$ but $\ep$ and $\delta$ remain fixed, and the condition \eqref{expcond} is satisfied, then the mixing time of the $Q$-chain is of order $\log n$. 

Let us now try to understand the condition \eqref{expcond}. The set $\me(A)$ is the `one-step expansion' of $A$, that is, the set of all states that are accessible in one step by the $P$-chain if it starts in $A$. By the condition that $p_{ii}>0$ for all $i$, we see that $\me(A)$ is a superset of $A$. The condition \eqref{expcond} says that for any $A$ of size $\le n/2$, if we apply the one-step expansion, then apply $f$ and then again apply the one-step expansion, the resulting set must be larger than $A$ by a fixed factor.


We found it difficult to produce explicit examples of functions that satisfy \eqref{expcond} with $\ep$ independent of $n$. Surprisingly, the following theorem shows that they are extremely abundant, even in our general setting, thereby answering both the questions (b) and (c) posed earlier.
\begin{thm}\label{thm2}
Let all notation be as in Theorem \ref{thm1}. There are universal  constants $C_1, C_2>0$ and $C_3\in (0,1)$ such that all but a fraction $C_1 n^{-C_2/\delta}$ of bijections $f$ satisfy \eqref{expcond} with $\ep = C_3 \delta$. Consequently, for all but a fraction $C_1 n^{-C_2/\delta}$ of bijections $f$, the corresponding $Q$-chain satisfies, for all $i\in S$ and $k\ge 1$,
\[
\|\mu_i^k- \mu\|_{TV} \le \frac{\sqrt{n}}{2} \biggl(1-\frac{C_3^2\delta^{10}}{2}\biggr)^{(k-2)/4}.
\]
\end{thm}
For instance, if the $P$-chain is the lazy random walk on $\zz_n$ with equal probabilities of taking one step to the left, or one step to the right, or not moving, then $\delta = 1/3$. So in this example this result shows that for all but a fraction $C_1 n^{-3C_2}$ of bijections of $\zz_n$, the corresponding $Q$-chain mixes in time $O(\log n)$. This is a striking improvement over the lazy random walk, which mixes in time $O(n^2)$.

Interestingly, the doubling random walk discussed in Section \ref{intro} {\it does not} satisfy the expansion condition \ref{expcond} with $\ep$ independent of $n$. To see this, take $n=4m-1$ for large enough $m$, and define
\[
A = \{1,2,\ldots,m-1\} \cup \{2m+1,\ldots, 3m-1\}.
\]
Then
\[
\me(A) = \{0,1,\ldots, m\} \cup \{2m, \ldots, 3m\}.
\]
So if $f$ is the doubling map, an easy calculation shows that
\[
f\circ \me(A) = \{0,1,2,\ldots, 2m+1\},
\]
and therefore
\[
\me\circ f\circ \me(A) = \{0,1,\ldots, 2m+2, n-1\}.
\]
So we have $|A| = 2m-2\le n/2$ and $|\me\circ f\circ\me(A)| = 2m+4$, which means that the condition \eqref{expcond} cannot hold with $\ep > 6/(2m-2)$. Therefore the speedup of the doubling walk is {\it not} happening due to expansion, although as Theorem \ref{thm2} shows, expansion is the reason for speedup for ``most'' bijections. 


\section{Related previous work}\label{review}
The deterministic doubling example comes from work of \citet{cdg87}. They found that for almost all $n$, $1.02\log_2 n$ steps suffice. \citet{hildebrand} showed that $1.004\log_2n$ steps are not enough for mixing. In a recent tour-de-force, \citet{ev20} show that for almost all $n$, there is a cutoff at $c\log_2n$ where $c = 1.01136...$. Two companion papers by \citet{bv19, bv18} relate this problem to irreducibility of random polynomials and very deep number theory. All of these authors consider $X_{k+1}=aX_k+\ep_{k+1}\pmod n$ for general $a$ and $\ep_i$'s having more general distributions. Related work is in \citet{hildebrand} and the Ph.D.~thesis of \citet{neville}. One natural extension is that if $G$ is a finite group with $\ep_1,\ep_2,\ldots$ i.i.d.~from a probability $\mu$ on $G$, the random walk $X_{k+1}=X_k\ep_{k+1}$ may be changed to $X_{k+1}=A(X_k)\ep_{k+1}$ where $A$ is an automorphism of $G$. This is studied in \citet{dg1, dg2} who treat $G = \zz_p^n$ with a matrix as automorphism. Similar speedups occur. 

A very rough heuristic to explain the speedup in our Theorem \ref{thm2} goes as follows. Composing the original random walk step with a random bijection mimics a jump along the edge of a random graph. Since random graphs are usually expanders, and random walks on expanders with $n$ vertices mix in $O(\log n)$ steps, we may naturally expect something like Theorem \ref{thm2} to be true. Remarkable work has been done on rates of convergence of random walks on $d$-regular expanders. For instance, \citet{lubetzkysly} show that random walks on $d$-regular random graphs have a cutoff for total variation mixing at $\frac{d}{d-2}\log_{d-1}n$.  \citet{lubetzkyperes} prove a similar result for Ramanujan graphs, with cutoff at $\frac{d}{d-1}\log_{d-1}n$. Related work has recently appeared in \citet{sousi}, who show that if you begin with any bounded degree graph on $2n$ vertices (for which all connected components have size at least $3$) and add a random perfect matching, the resulting walk has a cutoff and mixes in time $O(\log n)$. 

A setup very similar to ours, with different results, has been considered recently by \citet{bqz18}. This paper studies a random bistochastic matrix of the form $\Pi P$, where $\Pi$ is a uniformly distributed permutation matrix and $P$ is a given bistochastic matrix. Under sparsity and regularity assumptions on $P$, the authors prove that the second largest eigenvalue of $\Pi P$ is essentially bounded by the normalized Hilbert--Schmidt norm of $P$. They apply the result to random walks on random regular digraphs.

A different, related, literature treats piecewise deterministic Markov processes. Here, one moves deterministically for a while, then at a random time (depending on the path) a random big jump is made. This models the behavior of biological systems and chemical reactions. A splendid overview is in \citet{mal15} (see also \citet{benaim}). There is a wide swath of related work in a non-stochastic setting where various extra ``stirring'' processes are shown to speed up mixing and defeat diffusive behavior. See \citet{constantin}, \citet{ottino} and \citet{rallabandi}. The papers of \citet{dingperes}, \citet{hermon18} and \citet{hermonperes} are Markov chain versions of these last results. 

There are a host of other procedures that aim at defeating diffusive behavior. For completeness, we highlight four:
\begin{itemize}
\item Event chain Monte Carlo.
\item Hit and run.
\item Lifting/non-reversible variants.
\item Adding more moves.
\end{itemize}
Each of these is in active development; searching for citations to the literature below should bring the reader up to date.

While ``random walk'' type algorithms are easy to implement, it is by now well known that they suffer from ``diameter$^2$'' behavior. To some extent that remains true for general reversible Markov chains. Indeed, \citet{neal} showed that the spectral gap of {\it any} reversible Markov chain can be improved by desymmetrizing. Of course, the spectral gap is a crude measure of convergence. A host of sharpenings are in \citet{diaconismiclo}.

Non-reversible versions of Hamiltonian Monte Carlo introduced in \citet{dhn00} have morphed into a variety of ``lifting algorithms'' culminating in the event chain algorithm of \citet{krauth}. This revolutionized the basic problem of sampling hard disk configurations in statistical mechanics. The hard disk problem was the inspiration for three basic tools: The Metropolis algorithm, the Gibbs sampler (Glauber dynamics) and molecular dynamics.  All gave wrong conclusions for the problems of interest; the new algorithms showed that entirely new pictures of phase transitions are needed. It is an important open problem to generalize these algorithms to more general stationary distributions. 

Recent, closely related work is in \citet{bordenavelacoin} and \citet{conchonkerjan}. Written independently, both papers study `lifting' random walks to covering graphs to speed up mixing and both prove that most liftings have cutoffs.  The first paper also studies non-reversible chains. The paper of \citet{gerencserhendrickx} adds `long distance edges' (as in small world graphs)  and non-reversibility to get speedups. It too has close connections with the present project. 


A final theme is ``hit and run'': Taking long steps in a chosen direction instead of just ``going $\pm1$''. For a survey, see \citet{andersondiaconis}. It has been difficult to give sharp running time estimates for these algorithms (even though they are ``obviously better''). For a recent success, see the forthcoming manuscript of  \citet{bsc20}. 

It is natural to try to speed up convergence by adding more generators.  This is a largely open problem. For example, consider the simple random walk on the hypercube $\zz_2^d$. The usual generating set is $\{0,e_1,\ldots,e_d\}$, where $e_i$ is the $i^{\textup{th}}$ standard basis vector. At each step, one of these vectors is added with probability $1/(d+1)$. It is well known that $\frac{1}{4}d\log d +O(d)$ steps are necessary and sufficient for mixing (\citet{dsh}). Suppose you want to add more generators (say $2d+1$ generators) --- what should you choose? \citet{wilson97} studied this problem and showed that almost all generating sets of size $2d+1$ mix in time $0.2458 d$ steps. Mirroring our results above, his methods do not give an explicit choice effecting this improvement. 

These problems are interesting, and largely open, even for the cyclic group $\zz_p$ when $p$ is a prime. As above, the generating set $\{0,\pm1\}$ take order $p^2$ steps to mix. On the other hand, almost all sets of $3$ generators mix in order $p$ steps. \citet{green} found $3$ generators (roughly, $1$, $p^{1/3}$ and $p^{2/3}$) that do the job. References to this literature can be found in a survey by \citet{hilde}. 

There has been vigorous work on adding and deleting random edges at each stage (dynamical configuration model). See the paper of \citet{avena}, which also has pointers to another well studied potential speedup --- the non-backtracking random walk. 

Finally we note that natural attempts at speeding things up may fail. Cutting cards between shuffles and systematic versus random scan in the Gibbs sampler are examples. For details, see \citet{diaconis}. 

\section{Proof sketch}\label{sketchsec}
For the proof of Theorem \ref{thm1}, we define an auxiliary chain with kernel $R = L^2 (L^T)^2$, where $L := P\Pi$. This kernel is symmetric and positive semidefinite, and defines a reversible Markov chain. The proof has two steps --- first, relate the rate of convergence of the $R$-chain with that of the $Q$-chain, and second, use the reversibility of the $R$-chain to invoke Cheeger's bound for the spectral gap. A lower bound on the Cheeger constant is then obtained using condition \eqref{expcond}.

The proof of Theorem \ref{thm2} has its root in a beautiful idea from a paper of \citet{gangulyperes} (see also the related earlier work of \citet{pymarsousi}).  In our context, the idea translates to showing the following: Let $f$ be a uniform random bijection. First, show that ``most'' sets $A$ of size $\le 3n/4$ have the property that $|\me(A)|\ge (1+\ep)|A|$ for some suitable fixed $\ep>0$. Since $|\me(B)|\ge |B|$ and $|f(B)|=|B|$ for all $B$, this implies that \eqref{expcond} is satisfied for ``most'' $A$ of size $\le n/2$. But we want \eqref{expcond} to hold for {\it all} $A$ of size $\le n/2$. The Ganguly--Peres idea is to circumvent this difficulty in the following ingenious way. Since we know that $|\me(B)|< (1+\ep)|B|$ for only a small number of $B$'s of size $\le 3n/4$, it is very unlikely that for any given set $B$, $f(B)$ has this property. Consequently, for a given set $A$ such that $|\me(A)|\le 3n/4$, it is very unlikely that $f(\me(A))$ has this property (taking $B=\me(A)$). Thus, for any such set $A$, it is very likely that
\[
|\me\circ f\circ \me(A)|\ge (1+\ep)|f(\me(A))|\ge (1+\ep)|A|.
\]
Since the above event is very likely for any given $A$, it remains very likely if we take the intersection of such events  over a suitably small set of $A$'s. So we now take this small set to be precisely the set of $A$'s that did not satisfy $|\me(A)|\ge (1+\ep)|A|$ --- that is, our rogue set that spoiled \eqref{expcond}. By some calculations, we can show that this indeed happens, and so \eqref{expcond} holds for all $A$ with $|A|\le n/2$. 

In \cite{gangulyperes}, the Ganguly--Peres idea was applied to solve a different problem for a random walk on points arranged in a line. As noted above, the key step is to show that the set of all $A$ of size $\le 3n/4$ that violate $|\me(A)|\ge (1+\ep)|A|$ is a small set. Since we are working in a general setup here, the proof of this step from \cite{gangulyperes} does not generalize; it is proved here by a new method. 

\section{Proof of Theorem \ref{thm1}}\label{proofsec}
Let $L := P\Pi$. Since the uniform distribution on $S$ is the stationary measure of $P$, we deduce that $P$ is a doubly stochastic matrix. Since $\Pi$ is a permutation matrix, it is automatically doubly stochastic. Since the product of doubly stochastic matrices is doubly stochastic, we see that $L$ is also doubly stochastic. Thus, $L$ and $L^T$ are both stochastic matrices, and so is any power of either of these matrices. Consequently, the matrix 
\[
R := L^2 (L^T)^2 = P\Pi P\Pi\Pi^T P^T \Pi^TP^T = P\Pi PP^T \Pi^T P^T.
\]
is a stochastic matrix.  (This identity explains why we symmetrize $L^2$ instead of $L$. Indeed, $LL^T = P\Pi \Pi^T P^T = PP^T$, and the effect of $\Pi$ disappears.) 
\begin{lmm}\label{rproplmm}
The Markov chain with transition matrix $R$ is irreducible, aperiodic and reversible, with uniform stationary distribution. Moreover, $R$ is symmetric and positive semidefinite. 
\end{lmm}
\begin{proof}
It is evident from the definition of $R$ that $R$ is symmetric and positive semidefinite. Since $R$ is symmetric, the Markov chain with transition matrix $R$ is reversible with respect to the uniform distribution on the state space.


Since $p_{ii}>0$ for any $i\in S$, it follows that
\begin{align}\label{rposit}
r_{ii} \ge p_{ii} \pi_{if(i)}p_{f(i)f(i)} p_{f(i)f(i)} \pi_{if(i)} p_{ii} >0.
\end{align}
Thus, $R$ is aperiodic. Next, note that if $p_{ij}>0$, then  
\begin{equation}\label{rlower}
r_{ij} \ge p_{ij} \pi_{jf(j)}p_{f(j)f(j)} p_{f(j)f(j)} \pi_{jf(j)} p_{jj} >0.
\end{equation}
Since $P$ is irreducible, this shows that $R$ is also irreducible.
\end{proof}

\begin{cor}\label{rcor}
The principal eigenvalue of $R$ is $1$, the principal eigenvector is the vector of all $1$'s (which we will henceforth denote by $1$), and the principal eigenvalue has multiplicity one. The second largest eigenvalue of $R$ (which we will henceforth call $\lambda_2$), is strictly less than $1$. All eigenvalues of $R$ are nonnegative.
\end{cor}
\begin{proof}
Since the Markov chain defined by $R$ is irreducible and aperiodic (on a finite state space), and $R$ is symmetric, all but the last claim follow from the general theory of Markov chains. The eigenvalues are nonnegative because $R$ is a positive semidefinite matrix.
\end{proof}
The information about $R$ collected in the above corollary allows us to prove the next lemma.
\begin{lmm}\label{qdecay}
For any $k\ge 1$, and any $x\in \rr^n$ that is orthogonal to the vector $1$, 
\[
\|L^k x\|\le \lambda_2^{(k-1)/4}\|x\|.
\]
\end{lmm}
\begin{proof}
Take any vector $x\in \rr^n$  that is orthogonal to $1$. By the properties of $R$ given in Corollary \ref{rcor}, it is easy to see that 
\[
x^T R x \le \lambda_2 \|x\|^2. 
\]
But notice that
\begin{align*}
x^T R x =  x^T L^2 (L^T)^2 x = \|(L^T)^2 x\|^2. 
\end{align*}
So for any $x$ that is orthogonal to $1$, 
\begin{align}\label{contrac}
\|(L^T)^2 x\|^2\le \lambda_2 \|x\|^2. 
\end{align}
But if $x$ is orthogonal to $1$, then the stochasticity of $L$ implies that 
\[
1^T (L^T)^2 x = 1^T x = 0.
\]
Thus, $(L^T)^2x$ is also orthogonal to $1$. So we can apply \eqref{contrac} iteratively to get
\[
\|(L^T)^{2k}x\|^2 \le \lambda_2^k\|x\|^2
\]
for any positive integer $k$. But $L^T$, being a stochastic matrix, is an $\ell^2$-contraction. Therefore 
\[
\|(L^T)^{2k+1}x\|^2 \le \|(L^T)^{2k}x\|^2 \le \lambda_2^k\|x\|^2.
\]
This completes the proof of the lemma.
\end{proof}
Lemma \ref{qdecay} yields the following corollary, which relates the rate of convergence of the $Q$-chain with the spectral gap of the $R$-chain.
\begin{cor}\label{qdecaycor}
For any $k\ge 2$, and any $x\in \rr^n$ that is orthogonal to the vector $1$, 
\[
\|Q^k x\|\le \lambda_2^{(k-2)/4}\|x\|.
\]
\end{cor}
\begin{proof}
Note that $Q^k = (\Pi P)^k = \Pi L^{k-1} P$. Being stochastic matrices, $\Pi$ and $P$ are both $\ell^2$-contractions. Moreover, if $1^Tx =0$, then $1^T Px = 1^Tx=0$ since $P$ is doubly stochastic. Therefore by Lemma \ref{qdecay},
\begin{align*}
\|Q^k x\| &= \|\Pi L^{k-1} Px\| \le \|L^{k-1} Px\|\\
&\le \lambda_2^{(k-2)/4} \|Px\|\le \lambda_2^{(k-2)/4}\|x\|.
\end{align*}
This completes the proof of the corollary.
\end{proof}
Lastly, we need the following upper bound on $\lambda_2$.
\begin{lmm}\label{cheeger}
Let $\lambda_2$ be the second largest eigenvalue of $R$. Then 
\[
\lambda_2 \le 1-\frac{\ep^2\delta^8}{2},
\]
where $\ep$ is the constant from condition \eqref{expcond}.
\end{lmm}
\begin{proof}
Define a probability measure $\nu$ on $S\times S$ as
\[
\nu(i,j) = \mu(i) r_{ij} = \frac{r_{ij}}{n}. 
\]
Recall that the Cheeger constant for the Markov chain with transition matrix $R$ is 
\[
\Phi = \min_{A\subseteq S,\, \mu(A)\le 1/2} \frac{\nu(A\times A^c)}{\mu(A)},
\]
where $A^c := S\setminus A$.  Since $\mu$ is the uniform distribution on $V$, this simplifies to
\[
\Phi = \min_{A\subseteq S,\, |A|\le n/2} \frac{1}{|A|}\sum_{i\in A,\, j\in A^c} r_{ij}.
\]
Now, if $r_{ij}>0$, then from the inequality \eqref{rlower} and the assumption that $p_{kl}\ge \delta$ for all nonzero $p_{kl}$, we get $r_{ij} \ge \delta^{4}$. Therefore, if $A'$ is the set of all vertices that are attainable in one step from some $i\in A$ by the Markov chain with transition matrix $R$, then
\begin{align}\label{sumuv}
\sum_{i\in A,\, j\in A^c} r_{ij}\ge \delta^{4} |A'\cap A^c| =  \delta^4 (|A'|- |A|),
\end{align}
where the last identity holds because $A\subseteq A'$ (by \eqref{rposit}). Now, since $p_{ij}>0$ if and only if $p_{ji}>0$, the set $A'$ defined above can be written as
\[
A' = \me \circ f^{-1}\circ \me\circ \me \circ f\circ \me(A).
\]
Application of $\me$ cannot decrease the size of a set, and application of $f^{-1}$ does not alter the size. Therefore by condition \eqref{expcond},
\[
|A'|\ge |\me\circ f\circ \me(A)|\ge (1+\ep) |A|,
\]
provided that $|A|\le n/2$. Plugging this into \eqref{sumuv}, we get $\Phi \ge \ep\delta^4$. By the well-known bound
\[
\lambda_2\le 1-\frac{\Phi^2}{2},
\]
this completes the proof of the lemma.
\end{proof}
We are now ready to complete the proof of Theorem \ref{thm1}
\begin{proof}[Proof of Theorem \ref{thm1}]
Now take any $i\in S$ and $k\ge 2$. Let $x\in \rr^n$ be the vector whose $i^{\textup{th}}$ component is $1-1/n$ and the other components are all equal to $-1/n$. Then it is easy to see that $x$ is orthogonal to $1$, and that $(Q^T)^k x$ is the vector whose $j^{\textup{th}}$ component is  $\mu_i^k(j) - \mu(j)$. Thus by Corollary  \ref{qdecaycor},
\begin{align*}
\|\mu_i^k- \mu\|_{TV} &= \frac{1}{2}\sum_{j\in S} |\mu_i^k(j) - \mu(j)|\\
&\le  \frac{1}{2}\biggl(n\sum_{j\in S} (\mu_i^k(j) - \mu(j))^2\biggr)^{1/2}\\
&= \frac{\sqrt{n}}{2} \|(Q^T)^k x\|\le \frac{\sqrt{n}}{2}  \lambda_2^{(k-2)/4} \|x\|.
\end{align*}
Since $\|x\|\le 1$, this proves that
\begin{align}\label{tvbd}
\|\mu_i^k- \mu\|_{TV} &\le \frac{\sqrt{n}}{2}  \lambda_2^{(k-2)/4}.
\end{align}
Using the bound on $\lambda_2$ from Lemma \ref{cheeger}, this completes the proof when $k\ge 2$. When $k=1$, the bound is greater than $1$ and therefore holds trivially.
\end{proof}

\section{Proof of Theorem \ref{thm2}}\label{proof2}
For this proof, it will be convenient to define a graph structure on $S$. Join two distinct points (vertices) $i,j\in S$ by an edge whenever $p_{ij}>0$. Since $p_{ij}>0$ if and only if $p_{ji}>0$, this defines an undirected graph on $S$. The irreducibility of the $P$-chain implies that this graph is connected. A {\it path} in this graph is simply a sequence of vertices such that each is connected by an edge to the one following it.

Define the {\it external boundary} of a set $A\subseteq S$ as
\[
\partial A := \me(A)\setminus A.
\]
The following lemma proves a crucial property of the graph defined above.
\begin{lmm}\label{bdrylmm}
Take any $A\subseteq S$. Then the number of sets $B$ such that $\partial B = A$ is at most $2^{|A|/\delta}$.
\end{lmm}
The proof of this lemma is divided into several steps. Fix a set $A$. Define an equivalence relation on $A^c$ as follows. Say that two vertices $i,j\in A^c$ are equivalent, and write $i\sim j$, if either $i=j$, or there is a path from $i$ to $j$ that avoids the set $A$ (meaning that there is no vertex in the path that belongs to $A$). It is easy to see that this is an equivalence relation, because it is obviously reflexive and symmetric, and if $i\sim j$ and $j\sim k$, then there is a path from $i$ to $k$ that avoids $A$. To prove Lemma \ref{bdrylmm}, we need to prove two facts about this equivalence relation.
\begin{lmm}\label{bdlmm1}
If $B$ is a set of vertices such that $\partial B = A$, then $B$ must be the union of some equivalence classes of the equivalence relation defined above.
\end{lmm}
\begin{proof}
Take any $B$ such that $\partial B = A$. Take any $i\in B$ and $j\in A^c$ such that $i\sim j$. To prove the lemma, we have to show that $j\in B$. 

Let $m$ be the length of the shortest path from $i$ to $j$ that avoids $A$. Since $i\sim j$, we know that there is at least one such path, and therefore $m$ is well-defined. We will now prove the claim that $j\in B$ by induction on $m$. If $m=1$, then $j$ is a neighbor of $i$. Since $i\in B$, $j$ is a neighbor of $i$ and $j\not\in\partial B$, the only possibility is that $j\in B$. This proves the claim when $m=1$.

Let us now assume that the claim holds for all $m'<m$. Take a path of length $m$ from $i$ to $j$ that avoids $A$. Let $k$ be the vertex on this path that comes immediately before $j$. Then $k\in A^c$, and $k$ can be reached from $i$ by a path of length $m-1$ that avoids $A$. Therefore by the induction hypothesis, $k\in B$. Applying the case $m=1$ to $k$, we see that $j$ must be a member of $B$. This completes the induction step.
\end{proof}

\begin{lmm}\label{bdlmm2}
If $D$ is any equivalence class of the equivalence relation defined above, then there is at least one element of $D$ that is adjacent to some element of $A$.
\end{lmm}
\begin{proof}
Take any $i\in D$ and $j\in A$. Since the  graph is connected, there is a path from $i$ to $j$. Along this path, let $k$ be the first vertex that belongs to $A$. (This is well-defined because at least one vertex of the path, namely $j$, is in $A$.) Let $l$ be the vertex that immediately precedes $k$ in the path. Then we see that there is a path from $i$ to $l$ that avoids $A$, and so $i\sim l$. Therefore $l\in D$ and $l$ is adjacent to an element of $A$, which proves the lemma.
\end{proof}
Lastly, we need an upper bound on the maximum degree of our graph on $S$.
\begin{lmm}\label{maxdeg}
The maximum degree of the graph defined on $S$ is at most $1/\delta$.
\end{lmm}
\begin{proof}
Since $p_{ij}\ge \delta$ for every $i$ and $j$ that are connected by an edge, and 
\[
\sum_{j\in S} p_{ij}=1
\]
because $P$ is a stochastic matrix, the number of neighbors of $i$ must be $\le 1/\delta$. 
\end{proof}
We are now ready to prove Lemma \ref{bdrylmm}.
\begin{proof}[Proof of Lemma \ref{bdrylmm}]
Since each equivalence class of our equivalence relation has at least one element that is adjacent to an element of $A$ (by Lemma~\ref{bdlmm2}), the number of equivalence classes can be at most the size of $\partial A$. But by Lemma \ref{maxdeg}, $|\partial A|\le |A|/\delta$. Thus, there are at most $|A|/\delta$ equivalence classes.

But by Lemma \ref{bdlmm1}, if $B$ is any set such that $\partial B = A$, then $B$  must be a union of equivalence classes. Consequently, the number of such $B$ is at most the size of the power set of the set of equivalence classes. By the previous paragraph, this is bounded above by $2^{|A|/\delta}$.
\end{proof}
Lemma \ref{bdrylmm} has the following corollary, which shows that not too many sets can have a small external boundary.
\begin{cor}\label{mainbd}
Take any $1\le k\le n$. The number of sets $B\subseteq S$ such that $|\partial B|= k$ is at most ${n\choose k}2^{k/\delta}$.
\end{cor}
\begin{proof}
We can choose the external boundary of $B$ in at most ${n\choose k}$ ways. Having chosen the external boundary, Lemma \ref{bdrylmm} says that $B$ can be chosen in at most $2^{k/\delta}$ ways. Thus, the number of ways of choosing $B$ with the given constraint is at most ${n\choose k} 2^{k/\delta}$.
\end{proof}
For each $\ep>0$, let $\ma_{\ep}$ be the set of all $A\subseteq S$ such that $|\me(A)| < (1+\ep)|A|$. For each $1\le m\le n$, let $\ma_{\ep,m}$ be the set of all sets $A\in \ma_\ep$ that are of size $m$. Let
\[
\ma'_{\ep,m} := \bigcup_{1\le k\le m} \ma_{\ep,k}.
\]
The following lemma uses the bound from Corollary \ref{mainbd} to get a bound on the size of $\ma'_{\ep,m}$.
\begin{lmm}\label{aepbd}
For any $\ep>0$ and $1\le m\le n$,
\[
|\ma'_{\ep,m}|\le \sum_{1\le k< \ep m} {n\choose k}2^{k/\delta}.
\]
\end{lmm}
\begin{proof}
Take any $A\in \ma'_{\ep,m}$. Then note that 
\[
|\partial A|  = |\me(A)|-|A| < \ep |A|\le \ep m.
\]
The claim is now proved by Corollary \ref{mainbd}.
\end{proof}
Now fix some $\ep\in (0,1/2)$. Let $f$ be a {\it random} bijection chosen uniformly from the set of all bijections of $S$. Define an event
\[
E := \{f\circ\me(A)\in \ma_{\ep} \text{ for some $A\in \ma_\ep$ with $1/\ep \le |A|\le n/2$}\}.
\]
We want to show that $E$ is a rare event. The following lemma is the first step towards that.
\begin{lmm}\label{endlmm1}
Let $E$ be the event defined above. Then
\[
\pp(E) \le  \sum_{1/\ep \le m\le 3n/4}\frac{1}{{n\choose m}}\biggl(\sum_{1\le k< \ep m} {n\choose k}2^{k/\delta}\biggr)^2.
\]
\end{lmm}
\begin{proof}
Take any $A\in \ma_\ep$ with $1/\ep\le |A|\le n/2$. Let $B= \me(A)$. Then observe the following about $B$:
\begin{itemize}
\item Since $f$ is a uniform random bijection, the random set $f(B)$ is uniformly distributed among all subsets of size $|B|$.  
\item Since $A\in \ma_\ep$, $\ep\in (0,1/2)$, and $1/\ep\le |A|\le n/2$, 
\[
|B|<(1+\ep)|A|< \frac{3}{2}|A|\le \frac{3n}{4},
\]
and $|B|\ge |A|\ge 1/\ep$. 
\end{itemize}
These two observations imply that
\begin{align*}
\pp(E) &\le \sum_{1/\ep \le m\le 3n/4} \sum_{\substack{A\in \ma_\ep, \\ |\me(A)|=m}} \pp(f(\me(A))\in \ma_{\ep})\\
&= \sum_{1/\ep \le m\le 3n/4} \sum_{\substack{A\in \ma_\ep, \\ |\me(A)|=m}} \frac{|\ma_{\ep,m}|}{{n\choose m}}\\
&\le  \sum_{1/\ep \le m\le 3n/4}\frac{|\ma'_{\ep,m}||\ma_{\ep,m}|}{{n\choose m}}.
\end{align*}
Since $\ma_{\ep,m}'$ is a superset of $\ma_{\ep,m}$, plugging in the bound from Lemma \ref{aepbd} completes the proof.
\end{proof} 

The bound in Lemma \ref{endlmm1} is not straightforwardly understandable. The following lemma clarifies the matter.
\begin{lmm}\label{endlmm2}
There are universal constants $C_0$, $C_1$, $C_2$ and $C_3$ such that if $\ep\le C_0\delta$, then 
\[
\sum_{1/\ep \le m\le 3n/4}\frac{1}{{n\choose m}}\biggl(\sum_{1\le k< \ep m} {n\choose k}2^{k/\delta}\biggr)^2 \le C_1 e^{-C_2\sqrt{n}} + C_3^{1/\ep} n^{2-1/4\ep}.
\]
\end{lmm}
\begin{proof}
Recall the well-known inequalities
\[
\frac{n^k}{k^k} \le {n \choose k}\le \frac{e^k n^k}{k^k}.
\]
Also, check that the map $x \mapsto (en/x)^x$ is increasing in $[1,n]$. Thus,
\begin{align*}
&\sum_{1/\ep \le m\le 3n/4}\frac{1}{{n\choose m}}\biggl(\sum_{k< \ep m} {n\choose k}2^{ k/\delta}\biggr)^2\\
&\le \sum_{1/\ep \le m\le 3n/4} \frac{m^m}{n^m} n^2 \biggl(\frac{e n 2^{1/\delta}}{\ep m}\biggr)^{2\ep m}\\
&= n^2 \sum_{1/\ep \le m\le 3n/4}\biggl(\frac{m^{1-2\ep}}{n^{1-2\ep }}2^{2\ep/\delta } e^{2 \ep(1+\log (1/\ep))}\biggr)^m.
\end{align*}
Now choose $\ep$ so small that $1-2\ep \ge 1/2$ and 
\begin{align*}
2^{2\ep/\delta} e^{2 \ep(1+\log (1/\ep))} < \sqrt{\frac{8}{7}}.
\end{align*}
Note that this can be ensured by choosing $\ep$ to be less than some universal constant times $\delta$ (noting that $\delta\le 1$). With such a choice of $\ep$, we get
\begin{align*}
&n^2 \sum_{\sqrt{n} \le m\le 3n/4}\biggl(\frac{m^{1-2\ep}}{n^{1-2\ep }}2^{2 \ep/\delta } e^{2 \ep(1+\log (1/\ep))}\biggr)^m\\
&\le n^2 \sum_{\sqrt{n} \le m\le 3n/4}\biggl(\sqrt{\frac{3}{4}} \sqrt{\frac{8}{7}}\biggr)^m \\
&=  n^2 \sum_{\sqrt{n} \le m\le 3n/4}\biggl(\frac{6}{7}\biggr)^{m/2}\le C_1 e^{-C_2 \sqrt{n}},
\end{align*}
where $C_1$ and $C_2$ are universal  constants. On the other hand, if $1/\ep\le m\le \sqrt{n}$, then 
\begin{align*}
\frac{m^{1-2\ep}}{n^{1-2\ep }} \le n^{-1/4},
\end{align*}
and so
\begin{align*}
&n^2 \sum_{1/\ep\le m\le \sqrt{n}}\biggl(\frac{m^{1-2\ep}}{n^{1-2\ep }}2^{2\Delta \ep } e^{2 \ep(1+\log (1/\ep))}\biggr)^m\\
&\le n^2 \sum_{1/\ep\le m\le \sqrt{n}}\biggl(n^{-1/4} \sqrt{\frac{8}{7}}\biggr)^m\le C_3^{1/\ep} n^{2-1/4\ep}. 
\end{align*}
Adding up the two parts, we get the required bound.
\end{proof}
We are now ready to finish the proof of Theorem \ref{thm2}.
\begin{proof}
By Lemmas \ref{endlmm1} and \ref{endlmm2}, we see that there are universal constants $C_4,C_5>0$ and $C_6\in (0,1)$ such that if we choose $\ep = C_6\delta$, then 
\begin{align}\label{pebd}
\pp(E)\le C_4 n^{-C_5 /\delta}.
\end{align}
Suppose that $E$ does not happen. Then for any $A\in \ma_\ep$ with $1/\ep\le |A|\le n/2$, we have  $f(\me(A))\notin \ma_\ep$ and hence
\[
|\me\circ f\circ \me(A)|\ge (1+\ep) |f\circ \me(A)|\ge (1+\ep)|A|.
\]
On the other hand,  if $A\not\in \ma_\ep$, then
\[
|\me\circ f\circ\me(A)|\ge |\me(A)| \ge (1+\ep)|A|.
\]
Finally, if $|A|\le 1/\ep$ and $|A|\le n/2$, then since the $P$-chain is irreducible,
\[
|\me\circ f\circ\me(A)|\ge |\me(A)| \ge |A|+1 \ge (1+\ep) |A|.
\]
So if $E$ does not happen, then the random bijection $f$ satisfies the condition \eqref{expcond}.  By \eqref{pebd}, this completes the proof of Theorem \ref{thm2}.
\end{proof}

\section{A different speedup}\label{fibosec}
Going back to the original example (simple random walk on $\zz_n$), there is a different way to speed this up. Define a process $X_0,X_1,\ldots$ on $\zz_n$ by $X_0=0$, $X_1=1$ and 
\[
X_{k+1} = X_{k}+X_{k-1} +\ep_{k+1}\pmod n,
\]
where $\ep_i$ are independent, taking values $0$, $1$ and $-1$ with equal probabilities. Let $P_k(j) := \pp(X_k = j)$ and $U(j):=1/n$ for $j\in \zz_n$. 
\begin{thm}\label{fibothm}
For any $n\ge 22$ and $k = 5[(\log n)^2 + c\log n]$, 
\[
\|P_k- U\|_{TV} \le 1.6e^{-c/2}. 
\]
\end{thm}
{\it Remark.} The  best lower bound we have is that at least $\log n$ steps are required. It is natural to suspect that this is the right answer, but numerical experiments do not make a clear case. At any rate, we find it interesting that a simple recurrence speeds things up from order $n^2$ to order $(\log n)^2$. 

By running the recurrence, the chain can be represented as 
\begin{align*}
X_k = F_k + F_{k-1}\ep_2 + F_{k-2}\ep_3+\cdots F_1 \ep_{k} \pmod n,
\end{align*}
with $F_k$ the usual Fibonacci numbers $0,1,1,2,3,5,\ldots$ (so $F_5=5$). The (mod $n$) Fourier transform of $P_k$ is 
\begin{align}\label{fourier}
\hat{P}_k(a) = \ee(e^{2\pi i a X_k/n}) = e^{2\pi i a F_k/n} \prod_{b=1}^{k-1} \biggl(\frac{1}{3}+\frac{2}{3}\cos (2\pi a F_b/n)\biggr). 
\end{align}
We will use the inequality (see \citet[Chapter 3]{diaconisbook}) 
\begin{align}\label{fourierineq}
4\|P_k- U\|_{TV}^2 \le \sum_{a=1}^{n-1} |\hat{P}_k(a)|^2
\end{align}
to obtain an upper bound on the total variation distance between $P_k$ and $U$. We thus need to know about the distribution of Fibonacci numbers mod $n$. We were surprised that we couldn't find what was needed in the literature (see \citet{diacfibo}).  The following preliminary proposition is needed. 
Let $x_0,x_1,x_2,\ldots$ be any sequence of integers satisfying the Fibonacci recursion
\[
x_k = x_{k-1}+x_{k-2}.
\]
Take any $n$ such that at least one $x_i$ is not divisible by $n$. Let $b_k$ be the remainder of $x_k$ modulo $n$. We will prove the following property of this sequence.
\begin{prop}\label{mainprop}
For any $j$, there is some $j\le k \le j+8+3\log_{3/2}n$ such that $b_k\in [n/3, 2n/3]$.
\end{prop}
We need several lemmas to prove this proposition. 
\begin{lmm}\label{prelim1}
There cannot exist $k$ such that $b_k=b_{k+1}=0$.
\end{lmm}
\begin{proof}
If $b_k=b_{k+1}=0$ for some $k$, then $x_{k+1}$ and $x_k$ are both divisible by $n$. So the Fibonacci recursion implies that $x_j$ is divisible by $n$ for all $j$. But we chose $n$ such that at least one $x_i$ is not divisible by $n$. Thus, we get a contradiction.
\end{proof}
\begin{lmm}\label{claim1}
If $b_j,b_{j+1}\in [1,n/3)$ for some $j$, then there is some $j+2\le k< j+2\log_2n$ such that $b_k\in [n/3, 2n/3]$. 
\end{lmm}
\begin{proof}
Since the $b_i$'s satisfy the Fibonacci recursion modulo $n$, it follows that if $b_i$ and $b_{i+1}$ are both in $[1,n/3)$ for some $i$, then $b_{i+2} = b_{i+1}+b_i\in [1, 2n/3)$. So there exists an index $k$ which is the first index bigger than $j+1$ such that $b_k\in [n/3, 2n/3)$. We claim that for any $i\in [j+2,k]$, $b_i \ge 2^{(i-j)/2}$. To see this,  first note that for any $i\in [j+2,k]$, $b_i=b_{i-1}+b_{i-2}$. Therefore, since $b_j, b_{j+1}\ge 1$, the claim is true for $i=j+2$. Suppose that it is true up to $i-1$. Then
\begin{align*}
b_i &= b_{i-1}+b_{i-2}\ge 2^{(i-1-j)/2} + 2^{(i-2-j)/2}\\
&= 2^{(i-j)/2}(2^{-1/2} + 2^{-1}) \ge 2^{(i-j)/2}. 
\end{align*}
In particular, $b_k \ge 2^{(k-j)/2}$. But we know that $b_k < 2n/3$. Combining these two inequalities gives $k-j< 2\log_2 n$.
\end{proof}

\begin{lmm}\label{claim2}
If $b_j,b_{j+1}\in (2n/3,n-1]$ for some $j$, then there is some $j+2\le k< j+2\log_2n$ such that $b_k\in [n/3, 2n/3]$. 
\end{lmm}
\begin{proof}
Define $c_i := n - b_i$ for each $i$. Then the $c_i$'s also satisfy the Fibonacci recursion modulo $n$. Moreover, $c_j, c_{j+1}\in [1,n/3)$. Therefore by the proof of Lemma \ref{claim1}, there exist $k<j+2\log_2 n$ such that $c_k\in [n/3,2n/3]$. But this implies that $b_k\in [n/3,2n/3]$.
\end{proof}

\begin{lmm}\label{claim3}
If $b_j\in [1,n/3)$ and $b_{j+1}\in (2n/3,n-1]$ for some $j$, then there is some $j+2\le k< j+6+3\log_{3/2}n$ such that $b_k\in [n/3, 2n/3]$. 
\end{lmm}
\begin{proof}
For each $i$, let
\begin{align*}
d_i := 
\begin{cases}
b_i &\text{ if } b_i\in [0,n/2],\\
b_i - n &\text{ if } b_i \in (n/2, n-1].
\end{cases}
\end{align*}
Clearly, the $d_i$'s also satisfy the Fibonacci recursion modulo $n$, and $|d_i|\le n/2$ for each $i$. Take any $i$ such that 
\begin{itemize}
\item $d_i>0$, $d_{i+1}<0$,  $d_{i+2}>0$, $d_{i+3}<0$ and $d_{i+4}>0$, and 
\item $|d_i|$ and $|d_{i+1}|$ are less than $n/3$.
\end{itemize}
Under the above conditions, $d_i\in (0,n/3)$ and $d_{i+1}\in (-n/3,0)$. This implies that $d_i+d_{i+1} < d_i <n/3$ and $d_i+d_{i+1} > d_{i+1} > -n/3$. Thus, $|d_i+d_{i+1}|< n/3$. But we know that $d_{i+2}\equiv d_i+d_{i+1}\pmod n$, and $|d_{i+2}|\le n/2$. Thus, we must have that $d_{i+2}$ is actually equal to $d_i+d_{i+1}$, and therefore also that $|d_{i+2}|<n/3$.  Similarly, $|d_{i+3}|$ and $|d_{i+4}|$ are also less than $n/3$, and satisfy the equalities $d_{i+3}=d_{i+2}+d_{i+1}$ and $d_{i+4} = d_{i+3}+d_{i+2}$. Thus, 
\begin{align*}
0 &< d_{i+4} = 3d_{i+1}+2d_i,
\end{align*}
which gives 
\[
|d_{i+1}| = -d_{i+1} < \frac{2d_i}{3}= \frac{2|d_i|}{3}. 
\]
Similarly, if $d_i<0$, $d_{i+1}>0$, $d_{i+2}<0$, $d_{i+3}>0$ and $d_{i+4}<0$, and $|d_i|$ and $|d_{i+1}|$ are less than $n/3$, then also all of the absolute values are less than $n/3$, and 
\[
0 > d_{i+4}=3d_{i+1}+2d_i,
\]
which gives
\[
|d_{i+1}| = d_{i+1} < -\frac{2d_i}{3}=\frac{2|d_i|}{3}.
\]
Now let $j$ be as in the statement of the lemma. Then $d_j>0$, $d_{j+1}<0$, and $|d_j|$ and $|d_{j+1}|$ are both less than $n/3$. Let $l$ be an index greater than $j$ such that $d_j, d_{j+1},\ldots,d_l$ are all nonzero, with alternating signs. Suppose that $l\ge j+4$. The above deductions show that $|d_i|<n/3$ for all $i\in [j,l]$ and $|d_{i+1}|< 2|d_i|/3$ for all $i\in [j,l-4]$. Since $|d_{l-4}|\ge 1$ and $|d_j|<n/3$, this proves that $l$ cannot be greater than $j+4+\log_{3/2}n$. Thus, if we define $l$ to be the largest number greater than $j$ with the above properties, then $l$ is well-defined and is $\le j+4+\log_{3/2}n$.

By the definition of $l$, it follows that either $d_{l+1}=0$, or $d_{l+1}$ has the same sign as $d_l$. We already know that $d_l$ and $d_{l-1}$ are nonzero, have opposite signs, and are in $(-n/3,n/3)$. So, if $d_{l+1}=0$, then $d_{l+2}=d_{l+3}=d_l\in (-n/3,-1]\cup[1,n/3)$, and if $d_{l+1}$ is nonzero and has the same sign as $d_l$, then $|d_{l+1}|<n/3$. In the first situation, either both $b_{l+2}$ and $b_{l+3}$ are in $[1,n/3)$ or both are in $(2n/3,n-1]$. In the second situation, we can make the same deduction about $b_l$ and $b_{l+1}$. The claim now follows by Lemmas \ref{claim1} and \ref{claim2}.
\end{proof}

\begin{lmm}\label{claim4}
If $b_j\in (2n/3, n-1]$ and $b_{j+1}\in [1,n/3)$ for some $j$, then there is some $j+2\le k< j+6+3\log_{3/2}n$ such that $b_k\in [n/3, 2n/3]$. 
\end{lmm}
\begin{proof}
The proof is exactly the same as for Lemma \ref{claim3}.
\end{proof}
We are now ready to prove Proposition \ref{mainprop}.

\begin{proof}[Proof of Proposition \ref{mainprop}]
Take any $j$. If one of $b_j$ and $b_{j+1}$ is in $[n/3,2n/3]$, there is nothing to prove. If $b_j$ and $b_{j+1}$ are both in $[1,n/3)\cup(2n/3,n-1]$, then one of Lemmas \ref{claim1}--\ref{claim4} can be applied to complete the proof. If $b_j=0$ and $b_{j+1}\ne 0$, then $b_{j+2}=b_{j+1}\ne 0$, and so we can again apply one of the four lemmas. If $b_j\ne 0$ and $b_{j+1}=0$, then $b_{j+2}=b_{j+3} =b_{j}\ne 0$, and so again one of the four lemmas can be applied. Finally, note that by Lemma \ref{prelim1}, we cannot have $b_j=b_{j+1}=0$.
\end{proof}

Having proved Proposition \ref{mainprop}, we can now complete the proof of Theorem \ref{fibothm}. 
\begin{proof}[Proof of Theorem \ref{fibothm}]
By \eqref{fourier} and \eqref{fourierineq}, we get
\begin{align*}
4 \|P_k-U\|_{TV}^2 &\le \sum_{a=1}^{n-1} \prod_{b=1}^{k-1} \biggl(\frac{1}{3}+\frac{2}{3}\cos (2\pi a F_b/n)\biggr)^2.
\end{align*}
Now take any $1\le a\le n-1$. The sequence $aF_1, aF_2,\ldots$ satisfies the Fibonacci recursion, and the first term of the sequence is not divisible by $n$ since $a<n$. Thus, Proposition \ref{mainprop} is applicable to this sequence. Letting $m = 8+3\log_{3/2}n$, we get that at least $[(k-1)/m]$ among $aF_1,\ldots,aF_{k-1}$ are in $[n/3, 2n/3]$ modulo $n$. Now if $x\in [n/3,2n/3]$, then $\cos (2\pi x/n)\in [-1,-1/2]$, and so
\[
\frac{1}{3}+\frac{2}{3}\cos (2\pi a F_b/n) \in[-1/3, 0]. 
\]
Combining these observations, we get
\begin{align*}
4 \|P_k-U\|_{TV}^2  &\le n 9^{-[(k-1)/m]}. 
\end{align*}
It is easy to verify numerically that $30\le m\le 10\log n$ for $n\ge 22$, and also that $\log(9)/10\ge 1/5$. Thus, for $n\ge 22$ and $k= 5[(\log n)^2 + c\log n]$, 
\begin{align*}
4 \|P_k-U\|_{TV}^2  &\le n 9^{-(k-1)/m + 1}\\
&\le  n 9^{-k/m + 31/30} \le 9^{31/30}e^{-c}. 
\end{align*}
It can now be numerically verified that the claimed bound holds. 
\end{proof}

\section{Applications and open problems}\label{open}
One class of problems where uniform sampling is needed arises from exponential families.  Let $\mx$ be a finite set and $T:\mx \to\rr^d$ be a given statistic. For $\theta \in \rr^d$, let $p_\theta$ be the probability density
\[
p_\theta(x) = Z(\theta)^{-1} e^{\theta \cdot T(x)},
\]
where $Z(\theta)$ is the normalizing constant. The family $\{p_{\theta}\}_{\theta \in \rr^d}$ is called an exponential family of probability densities with sufficient statistic $T$. If $X\sim p_\theta$, the conditional distribution of $X$ given $T(X)=t$ is the uniform distribution on $\mx_t := \{x\in \rr^d: T(x)=t\}$. Such models appear in myriad statistical applications such as contingency tables and graphical models. They also appear in physics as Ising and related models. Uniform sampling on $\mx_t$ is required to test if a given dataset fits the model. An overview is in \citet{diaconissturmfels}, who introduced Gr\"obner basis techniques to do the sampling.  These are typically diffusive and it would be wonderful to have speedups. A second use for uniform sampling comes from drawing samples from the original $p_\theta$. For low-dimensional $T$, this can be done by sampling from the marginal distribution of $T$, and along the way, sampling $X$ given $T(X)=t$.  

Our main result shows that almost every bijection gives a speedup. This is not the same as having a specific bijection (such as $x\mapsto ax \pmod n$). Finding these, even in simple examples, seems challenging. One case where lots of bijections can be specified comes from permutation polynomials. Let us work mod $p$ for a prime $p$. Then a permutation polynomial is a polynomial $f$ with coefficients mod $p$ such that $j\mapsto f(j)$ is one-to-one mod $p$. Large classes of these are known. The Wikipedia entry is useful and the article by \citet{guralnick} shows how these can be found in several variables to map varieties (mod $p$) to themselves.

As an example, suppose that $(3, p-1)=1$. Then the map $j\mapsto j^3 \pmod p$ is one-to-one. The corresponding walk is
\[
X_{k+1} = X_k^3 + \ep_{k+1}\pmod p.
\]
We have no idea how to work with this but (weakly) conjecture that it mixes in order $\log p$ steps. 

Our colleague Kannan Soundararajan has suggested $f(0)=0$, $f(j) = j^{-1}\pmod p$ for $j\ne 0$. Preliminary exploration did not reveal this as an easy problem. 

It is natural to try to generalize the Fibonacci recurrence walk further. The following generalization was proved by Jimmy He, who kindly allowed us to include his result in this paper. 



Let $X$ be a finite set and $P$ be a Markov kernel on $X$. Let $f:X^n\to X$ be a function such that $f(\cdot,x_2,\dotsc, x_n)$ is a bijection for all $x_2,\dotsc, x_n\in X$. Define a Markov chain $P_f$ on $X^n$ by moving from the state $(X_1,\dotsc, X_n)$ to the state $(X_2,\dotsc, X_n, f(X_1,\dotsc, X_n))$, and then taking a step from $P$ in the last coordinate. Here the last coordinate can be viewed as a higher order Markov chain on $X$, which depends on the previous $n$ steps of the walk.

To see that this is a generalization of the Fibonacci walk, note that if $n=2$, $X=G$ is a finite Abelian group, $P$ is a random walk on $G$ generated by a measure $Q$ on $G$, and $f(x_1,x_2)=x_1+x_2$, then the walk moves from $(X_1,X_2)$ to $(X_2,X_1+X_2+\varepsilon)$ where $\varepsilon$ is drawn from $Q$, and this is exactly the original Fibonacci walk.

\begin{prop}[Due to Jimmy He, personal communication]\label{jimmy}
Assume that $P$ is lazy and ergodic, and has a uniform stationary distribution. Then $P_f$ is ergodic, and has a uniform stationary distribution.
\end{prop}
\begin{proof}
First, note that the assumption that $f(\cdot, x_2,\dotsc, x_n)$ is a bijection for all $x_2,\dotsc, x_n\in X$ implies that the function $g:X^n\to X^n$ defined by 
\[
g(x_1,\dotsc, x_n)=(x_2,\dotsc, x_n,f(x_1,\dotsc, x_n))
\]
is a bijection. Then $P_f$ can be described as applying the function $g$, followed by a step from $P$ in the last coordinate. 

Now we describe how to mimic steps from $P$ using $P_f$. Note that since $g$ is a permutation, $g^m$ is the identity permutation for some $m$. Since $P$ is lazy, we can alternate applying $g$, and then remaining stationary when taking a step from $P$. Doing so $m-1$ times, and then applying $g$ one last time, we can then take a step from $P$ in the last coordinate. Thus, it is possible to move from $(x_1,\dotsc,x_{n-1}, x_n)$ to $(x_1,\dotsc, x_{n-1}, x_n')$ in $m$ steps of $P_f$, where $x_n'\in X$ is some state for which $P(x_n,x_n')>0$. We call this mimicking a step from $P$.

To show that $P_f$ is irreducible, we simply repeat the above procedure. Since $P$ is irreducible, we can reach any state in $X$ in the last coordinate, while keeping the first $n-1$ coordinates fixed. But now, we can apply $g$ once, and then repeat the above procedure, and now the last two coordinates can be made arbitrary (since the procedure described fixes the first $n-1$ coordinates). Repeating this $n$ times allows any state in $X^n$ to be reached.

Now we show that $P_f$ is aperiodic. Starting from any state, we can return in $m$ steps by mimicking a lazy step from $P$ using the above procedure. But we can also first take a single step from $P_f$, and then use the same procedure as in the proof of irreducibility to return to the initial state, using a path of length $km$ for some $k$. This gives a path from the state to itself of length $km+1$. As $m$ and $km+1$ are coprime, this implies that $P_f$ is aperiodic.

Finally, it is clear that the stationary distribution is uniform, since $g$ is a bijection and so preserves the uniform distribution, and taking a step from $P$ in the last coordinate also preserves the uniform distribution.
\end{proof}

Proposition \ref{jimmy} allows us to generalize the Fibonacci walk to finite non-Abelian groups. If $G$ is a finite group and $Q$ is a probability on $G$ with support generating $G$ and having nonzero mass at the identity, then Proposition \ref{jimmy} implies that the second order Markov chain  which proceeds by $X_{k+1} =X_k X_{k-1}\epsilon_{k+1}$ (with the $\epsilon_i$ drawn i.i.d.~from $Q$), is ergodic with uniform stationary distribution.

Proposition \ref{jimmy} also allows generalization of Fibonacci walks to nonlinear recurrences on finite fields. For example, take a prime $p$ such that $(3, p-1) = 1$, and consider the random walk on $\mathbb{F}_p$ which proceeds as 
\[
X_{k+1} = X_k^3 + X_{k-1} + \ep_{k+1},
\]
where $\ep_i$ are i.i.d.~uniform from $\mathbb{F}_p$. Proposition \ref{jimmy} implies that this has uniform stationary distribution, since $f(x,y) = x^3 + y$ is a bijection in $x$ for every fixed $y$. 

Getting rates of convergence in any of the above examples seems like a challenging problem.

\section*{Acknowledgments}
We thank Shirshendu Ganguly for bringing the paper \cite{gangulyperes} to our attention, which was crucial for  the proof of Theorem \ref{thm2}. We thank Kannan Soundararajan, Perla Sousi, Ron Graham, Fan Chung, Charles Bordenave, Jimmy He, Huy Pham, and especially Jonathan Hermon for many insightful comments and references. Lastly, we thank the referee for a number of useful comments.

\end{document}